\documentclass[12pt]{amsart}
\usepackage{amsmath}           
\usepackage{amssymb}           

\newcommand{\CO}[1]{}

\DeclareMathOperator{\im}{{\sf im}}
\DeclareMathOperator{\End}{{\sf End}}

\DeclareMathOperator{\lt}{L}

\DeclareMathOperator{\GL}{GL}

\DeclareMathOperator{\G}{G}
\DeclareMathOperator{\CG}{GC}

\DeclareMathOperator{\Eq}{Eq}

\DeclareMathOperator{\A}{A}

\DeclareMathOperator{\D}{{\sf D}}
\DeclareMathOperator{\id}{{\sf id}}

\DeclareMathOperator{\Q}{{\sf Q}}

\DeclareMathOperator{\latr}{{\sf L}}  
\DeclareMathOperator{\latn}{\overline{{\sf L}}}

\begin{document}

\newtheorem{thm}{Theorem}   
\newtheorem{lemma}[thm]{Lemma}   
\newtheorem{proposition}[thm]{Proposition}   
\newtheorem{theorem}[thm]{Theorem}   
\newtheorem{corollary}[thm]{Corollary}   
\newtheorem{pro}[thm]{Proposition}
\newtheorem{cor}[thm]{Corollary}
\newtheorem{lem}[thm]{Lemma}
\newtheorem{claim}[thm]{Claim}
\newtheorem{fact}[thm]{Fact}

\newcommand{\lab}[1]{\label{#1}}%\marginpar{\footnotesize #1}} 
\renewcommand{\labelenumi}{{\textrm{(\roman{enumi})}}}
\newcommand{\vep}{\varepsilon} 
\newcommand{\One}{1} 
\newcommand{\Zero}{0} 
  \newcommand{\mb}[1]{\mathbb{#1}}
  \newcommand{\mc}[1]{\mathcal{#1}}
  \renewcommand{\phi}{\varphi}
\newcommand{\Qm}{\Q \mc{M}_f}
\newcommand{\et}{\wedge}
\newcommand{\imp}{\rightarrow  }
 \newcommand{\ex}{\exists  }
 \newcommand{\all}{\forall  }
\newcommand{\Mod}{\sf Mod  }
\newcommand{\al  }{\alpha }
\newcommand{\be  }{\beta }
\newcommand{\ga  }{\gamma }
\newcommand{\de  }{\delta }
\newcommand{\Si}{\Sigma  }
\newcommand{\no}{\noindent}
\newcommand{\sub}{\subseteq}
\renewcommand{\th}{\theta}
\newcommand{\eqq}{\mb{E}}

%%%%%%%%%%%%%%%%%%%%%%%%%%%%%%%%%%%%%%%%%%%%%%%%%%%%%%%%%%%%%%%%%%%%%%%%%%%%%% 
\title[Consistency problem]{On the consistency
problem for  modular lattices and related structures}

\author[Christian Herrmann]{Christian Herrmann}
\address[Christian Herrmann]{Technische Universit\"{a}t Darmstadt FB4\\Schlo{\ss}gartenstr. 7\\64289 Darmstadt\\Germany}
\email{herrmann@mathematik.tu-darmstadt.de}

\author[Yasuyuki Tsukamoto]{Yasuyuki Tsukamoto}
\address[Yasuyuki Tsukamoto]{Department of Mathematics,
Faculty of Science,
Kyoto University,
Kitashirakawa Oiwake-cho, Sakyo-ku,
Kyoto 606-8502, Japan}
\email{tsukamoto@i.h.kyoto-u.ac.jp}

\author[Martin Ziegler]{Martin Ziegler}
\address[Martin Ziegler]{KAIST, School of Computing,
291 Daehak-ro, Yuseong-gu, 34141 Daejeon, Republic of Korea}
\email{ziegler@cs.kaist.ac.kr}

\begin{abstract}
The 
 consistency problem
for a class of algebraic structures asks for an algorithm to decide
for any given  conjunction of equations whether it admits a non-trivial
satisfying
assignment within some member of  the class.
By Adyan (1955) and Rabin (1958) 
it is known unsolvable for (the class of) groups and,
recently, by Bridson and Wilton (2015) for  finite groups.
We derive unsolvability for (finite)   
modular lattices and
various subclasses; in particular,
the class of all subspace lattices of finite dimensional
vector spaces over a fixed or arbitrary field of
characteristic $0$.
The lattice results are used to prove unsolvability of
the consistency problem for (finite) rings and (finite)
representable relation algebras.
These results  in turn apply to equations between simple
expressions in Grassmann-Cayley algebra 
and to  functional and embedded multivalued
dependencies in databases.
\end{abstract}


\maketitle

\section{Introduction}
A solution of the  \emph{consistency problem}
for a class $\mc{C}$ of structures and a set $\Sigma$ of constraints
consists in an algorithm which, given any $\varphi\in\Sigma$,
decides whether there is a structure $A\in\mc{C}$
and a non-trivial assignment in $A$ satisfying $\varphi$.
Here, in the context of a ixed set $\Sigma$,
an assignment is ``trivial'' if it satisfies 
\emph{all} constraints $\psi\in\Sigma$.
For classes of  algebraic structures,  the familiar
 constraints are 
conjunctions of equations.
In the case $\Sigma$ consists of all of them,
the complement of the
consistency  problem
is known
as the  \emph{triviality problem}:
to decide for a given conjunction of equations 
whether every satisfying assignment
within the class 
generates  a singleton subalgebra 
(that is, whether 
the associated finitely presented algebra is trivial)
---
a  problem  reducing to the word problem.
A famous  instance of unsolvability 
is  given by the class of all groups,
 Adyan \cite{adyan,adyan2} and Rabin \cite{rabin};
the case of all finite groups is due to Bridson and Wilton \cite{brid}.

We use these to show in Section~\ref{MOLs}
the consistency problem unsolvable for  classes
of  modular lattices 
and subclasses of the quasivariety generated by finite modular lattices;
these classes are supposed to satisfy  certain richness conditions
concerning  the presence of `sufficiently many'
   subspace lattices $\lt(V)$
of vector spaces $V$.

This applies, in particular,
to the class of all $\lt(V)$ where $V$ is finite dimensional
over a fixed field of characteristic $0$.
In case the $V$ are real or complex Hilbert spaces,
unsovability extends to the associated ortholattices (Section~\ref{CMOLs}),
thus giving a
negative answer to the question raised in \cite[\S III.C]{LiCS},
the decision problem for ``quantum satisfiability in indefinite
(yet finite) dimensions'' (see also \cite {jacm}).

Section~\ref{Frames} recalls our central tool: 
the interpretation via (von Neumann) frames,
translating group  presentations into such for modular lattices. 
This provides a reduction of word problems
if the encoding is also applied to the equations to be decided,
cf.   Lipshitz \cite{lip} and  Freese \cite{freese}.
However, such encoding turns trivial assignments in groups 
into non-trivial assignments within frame generated sublattices.
Thus, we devise in Lemma~\ref{5} of Section~\ref{MOLs} 
a bit more sophisticated encoding based on 
lattice relations specific for 
\emph{fixed-point free} actions of linear groups ---
after asserting in Lemma~\ref{gp} of Section~\ref{AlgebraicStructures}
faithful such representation to indeed exist

The methods and results of \cite{hcom,data} allow to transfer 
unsolvability of the consistency problem  to (finite) representable
 relation algebras (Section~\ref{RelationAlgebras}); further 
to (finite)  relational databases with conjunctions of functional and embedded multivalued
dependencies as constraints (Section~\ref{databases}). 
In particular, there is no algorithm to decide for every 
 finite conjunction of 
\emph{functional dependencies} and \emph{embedded multivalued dependencies}
whether it implies that all  attributes are keys.
Consistency    for conditional  inclusion and functional
dependencies has been studied in \cite{idep};
 undecidability has been shown for the
combination of both.

Using the description of joins and meets of principal
right ideals in regular rings, the consistency problem
for classes of (finite) modular lattice can be reduced to that for
classes of (finite) rings in Section~\ref{Rings}. 
The special case of endomorphism rings
gives a further reduction to satisfiability of conjunctions
of equations between simple expressions
in Grassmann-Cayley algebras (Section~\ref{GrassmanCayley}). 
Thus, these problems
turn out algorithmically unsolvable, too.

\section{Algebraic structures} \lab{AlgebraicStructures}
We consider classes $\mc{C}$ of  structures
and sets $\Sigma$ of constraints,
that is formulas $\pi$ in a language associated with $\mc{C}$
--- we write $\pi(\bar x)$ where
$\bar x$ denotes the list of free variables in $\pi$.
An \emph{assignment} within  $A \in \mc{C}$
is a map $\bar x \mapsto \bar a$
into $A$; it \emph{satisfies} $\pi(\bar x)$ if $A\models \pi(\bar a)$;
it is \emph{trivial}
if it satisfies all  constraints  $\pi(\bar x) \in \Sigma$.
A solution of the \emph{consistency problem} for $\mc{C}$ and $\Sigma$
 consists in an algorithm which
decides for any constraint whether there is
a non-trivial satisfying assignment
within $\mc{C}$, that is, within some $A\in \mc{C}$.

Primarily, we consider classes $\mc{C}$ of algebraic structures
of finite signature;
here,  the usual
constraints are conjunctions of equations.
If  $\Sigma$ consists of exactly  these,
we speak just of the  \emph{consistency problem} for $\mc{C}$;
 trivial assignments
are those within singleton subalgebras of some $A \in \mc{C}$;
if $\mc{C}$ is a class of rings with unit or bounded lattices,
$A$ must be trivial (requiring $0=1$).

Of course,
unsolvability  with respect to $\Sigma$  for $\mc{C}$ 
is  
inherited by any expansion $\mc{C}'$ of $\mc{C}$
(that is, the language of $\mc{C}'$ has some
operation symbols in addition to that of $\mc{C}$
and
the members of $\mc{C}'$ arise from those of $\mc{C}$
by adding  operations denoted by these  additional symbols).
But, 
trivial  assignments  may generate non-trivial subalgebras
in the expansion.
Though, if within $\mc{C}$ trivial assignments require
trivial algebras, 
  unsolvability 
of the consistency problem  is inherited
by any expansion.

A \emph{quasi-variety} 
is a class of algebraic structures definable by
sets of quasi-identities: sentences of the form
$\forall \bar x.\; \pi(\bar x) \Rightarrow \psi(\bar x)$
where $\pi$ and $\psi$ are conjunctions of equations,
$\pi$ being possibly empty. Given a class
$\mc{C}$, the smallest quasi-variety $\Q \mc{C}$
containing $\mc{C}$ (i.e. \emph{generated by} $\mc{C}$)
is the model class of the set of quasi-identities valid in $\mc{C}$.
 The consistency problems for 
$\mc{C}$ and $\Q \mc{C}$ are equivalent due to the following.

\begin{fact}\lab{tri}
A conjunction $\pi(\bar x)$ of equations admits a non-trivial
satisfying assignment in $\mc{C}$ if and only if it does so in 
$\Q \mc{C}$. 
\end{fact}
\begin{proof}
Consider the quasi-identity $\forall \bar x.\;\pi(\bar x) \Rightarrow
\psi(\bar x)$ where $\psi$ is the conjunction
of all equations $x_i=x_1$ and $f(x_1, \ldots x_1)=x_1$,
$f$ an operation symbol. This is valid in $A$
if and only if $A$ admits only trivial satisfying
assignments for $\pi(\bar x)$. 
\end{proof}

Recall that a \emph{positive primitive} formula
is of the form $\exists \bar x \;\alpha(\bar x)$
where $\alpha$ is a conjunction of  atomic formulas. 
By a \emph{basic equation}
we mean an equation of the form  $y=x$ or  $y=f(\bar x)$ where
$f$ is an operation symbol. 
An \emph{unnested pp-formula} is of the form  
$\exists \bar y. \phi(\bar x,\bar y)$
where $\phi(\bar x,\bar y)$ is a conjunction of basic equations.
For the following compare  \cite[Theorem 2.6.1]{hodges}. 

\begin{fact}\lab{p1}
Every  conjunction $\pi(\bar x)$  of equations 
is logically equivalent to an unnested pp-formula
 $\exists \bar y. \phi(\bar x,\bar y)$. 
Moreover, in the case of a (bounded) lattice $L$,  if 
$L\models \pi(\bar a)$ only for single valued $\bar a$
then $L\models \phi(\bar a,\bar b)$
only for single valued $\bar a,\bar b$.
\end{fact}

Unsolvability for the classes of structures
to be considered, here, is shown by   reducing 
from the following deep results of Adyan \cite{adyan,adyan2}, Rabin \cite{rabin},  and 
 Bridson and Wilton
\cite{brid}.
\begin{theorem}\lab{p0}
The consistency problems for the class of all
groups and the class of all finite groups are  unsolvable.
\end{theorem}

\begin{corollary} \lab{c0} 
Let $\mc{C}$ 
be a class of (finite)  semigroups or
monoids 
such that any (finite) group embeds into some member of $\mc{C}$.
Then the consistency problem for $\mc{C}$ is unsolvable.
\end{corollary}
In particular, we may consider groups just with multiplication.
\begin{proof}
In the case of monoids with unit $e$,
given a  
 conjunction $\pi(\bar x)$
of group equations, form the conjunction $\hat{\pi}(\bar x, \bar y)$ 
of $\pi(\bar x)$ and  the $x_iy_i=e=y_i
x_i$ with  new variables $y_i$.
Thus, $\pi(\bar x)$ admits a non-trivial assignment
within some (finite) group if and only $\hat{\pi}$
does so within   some (finite) member of $\mc{C}$, namely within the group of units.
In the absence of constant  $e$, mimic it   by a new variable
$u$ adding  the equations $x_iu=x_i=ux_i$, $y_iu=y_i=uy_i$.
\end{proof}

The following is the intermediate step when deriving a lattice from a 
group. Supposedly, it is well known.
To some extend it could be replaced by  use of
 Maschke's Theorem.
  For a vector space $V$
over a division ring $F$ of characteristic $c$ we write
$\chi(F)=\chi(V) =c$. 
Let $\mc{V}_F$ denote the class of all $F$-vector spaces.
A representation $\rho$ of $G$ in $V$ is
\emph{fixed point free} if
 $\rho(g)(v)=v$
for all $g \in G$ only if $v=0$.

\begin{lemma} \lab{gp}
Let $G$ be a group and $V$ a vector space
where either   $\dim V\geq |G|$ and $G$  is infinite
or $\dim V=|G|-1>0$ is finite 
and $\chi(V)$ does not divide $|G|$.
Then there exists a fixed point free 
faithful representation of $G$ in $V$.
\end{lemma}
\begin{proof}
For infinite $G$, and $\dim V=|G|$,
we use the regular representation: We  may assume  $G$ a basis of $V$
and define $\rho(g)$ given by the basis permutation
$h \mapsto gh$. The claim follows from the fact
that this action of $G$ on $G$ is transitive:
 $v\neq 0$ in $V$ has the form $v=\sum_{h \in H} r_h h$ 
with some finite $H \subseteq G$ and $r_h \neq 0$;
choosing $k \in G\setminus H$, there is $g \in G$ with $gh=k$
whence $gv= \sum_{h \in H} r_hgh \neq v$.
For $\dim V >|G|$ we use a suitable direct multiple
of this representation. 

For $G$ the $2$-element group $\{ e,g\}$,
define the action on $V$ by $gv=-v$.
For finite $G$ of order $>2$,
 again assuming $G$ a basis of $V$, we have the 1-dimensional
invariant subspace $U$ spanned by $\sum_{g \in G} g$
and the induced action of $G$ on $V/U$.
The $g \in G^+:=G\setminus\{e\}$ form a basis of $V/U$;
that is, any $v+U$ in $V/U$ has a unique representation
\[v+U = \sum_{h \in G^+} r_h h +U
=\bigl(\sum_{h \in G^+, gh\neq e} r_{gh} gh\bigr) +r_gg  +U
\]
for any $g \in G^+$, in particular,
\[e+U = \sum_{ h \in G^+} -h +U=\bigl(\sum_{h \in G^+, gh\neq e} -gh\bigr)- g +U.\]
Thus, 
for  $v+U=\sum_{h \in G^+} r_h h +U$ and $g \in G^+$ one has
$ g(v+U)=$ 
\[=\bigl(\sum_{h \in G^+, gh\neq e} r_hgh\bigr)  -r_{g^{-1}}e +U
=\bigl(\sum_{h \in G^+, gh\neq e} (r_h -r_{g^{-1}})gh\bigr) - r_{g^{-1}}g +U
\] 
and the last expression returns $g(v+U)$ as 
 a linear combination of basis vectors of 
$V/U$. 
Assume  $v+U=g(v+U)$ for all $g \in G$;
that is, for all $g,h \in G^+$ one has
$r_g=-r_{g^{ -1}}$ and
 \[ r_{gh} =   r_h -r_{g^{-1}}=r_h+r_g \mbox{ if } gh\neq e.\]
For each $h \in G^+$, 
it follows $r_{h^k}=kr_h$, by induction, for all $1\leq k<\ell$
where  $\ell$ is  the order of $h$;
in particular, 
$-r_h =r_{h^{-1}}=r_{h^{\ell-1}}=(\ell-1)r_h$,
whence $\ell r_h=0$ and $r_h=0$, due to the assumption on
the characteristic.  Thus, $v+U=0+U$.
\end{proof}

\section{Coordinates in modular lattices} \lab{Frames}

We consider 
lattices as algebraic structures  with
operations  join $a+b$ and meet $a \cap b$;
in particular,
with respect to a suitable partial order $\leq$,
one has
 $a+b =\sup\{a,b\}$ and $a \cap b =\inf\{a,b\}$.
A lattice is \emph{modular} if $a \geq b$ implies
$a\cap(b+c) =b +a \cap c$.
Let 
 $\mc{M}_f$ denote the class of 
all  finite modular lattices.
 The lattice of all equivalence 
relations on the set $S$
is denoted by
  $\Eq(S)$. A sublattice $L$ of the latter is a 
\emph{lattice  of permuting equivalences}   
if,  for the relational product 
$\alpha\circ\beta=\{(x,z):\exists y.(x,y)\in \alpha, (y,z)\in\beta\}$, 
of any  $\alpha,\beta \in L$,
one has  $\alpha \circ \beta=
\beta \circ \alpha$;
that is,  $\alpha \circ \beta$ is 
the join $\alpha + \beta$
 in $\Eq(S)$ and $L$  and, in particular, transitive.
In that case, $L$ is a modular lattice \cite{jon2}.
The lattices $\lt(V)$ 
of all linear subspaces of the vector space $V$
are isomorphic to such: associate with a subspace 
$U$ the equivalence relation on $V$ defined  by $x-y \in U$.
Bounds of a lattice, if considered as constants, will be denoted
 by $0$ (bottom) and $1$ (top); in case of $\lt(V)$
these are $\{0\}$ and $V$. 
 We write $a \oplus b=c$ if
$a+b=c$ and $a \cap b =0$.

An $n$-\emph{frame  in} a  lattice $L$  is a system   
 $\bar a$  of elements $a_1,\ldots ,a_n$, $a_{ij}=a_{ji}\;(1\leq
 i<j\leq n)$, and $a_\bot,a_\top$ of $L$ such that,  where $\sum_{i \in \emptyset} a_i := a_\bot$,
\[ (\sum_{i \in I}a_i) \cap \sum_{j \in J} a_j =\sum_{k \in I\cap J}a_k
\quad \mbox{ for } I,J \subseteq \{1,\ldots ,n\},   \] 
  $a_\top=\sum_\ell a_\ell$, 
and, for pairwise distinct $i,j,k$,
\[  a_i+a_j =a_i+ a_{ij},\;\;
 a_i \cap a_{ij} =a_\bot,\;\;
a_{ik}=(a_i+a_k)\cap(a_{ij}+a_{jk}). \]
Define 
 \[\G(L,\bar a)=\{g \in L\mid g+ a_1=g+ a_2
=a_1+a_2,\; g\cap a_1=g \cap a_2 =a_\bot \}.\]
If $L$ has bounds $0,1$ and if $a_\bot=0$ and 
$a_\top=1$  then we speak of an $n$-\emph{frame of} $L$.
We use  $\bar z$  to denote a system of variables 
to be interpreted  by $4$-frames. Items  (i)--(iii)(b) of the following are 
 well known
in a broader context
\cite{neu,art,lip,freese}; our modest amendment of Item~(iv) 
will turn out as crucial to establish Lemma~\ref{5} below.
All  can be generalized to any fixed
$n \geq 4$.  We state and prove what
is relevant, here. We say that a subgroup $G$ of
$\GL(V)$ \emph{acts fixed point free} if
$gv=v$ for all $g\in G$ only if $v=0$.

\begin{lemma} \lab{fex}
\begin{enumerate} 
\item
For any  vector space $V$ with  $\dim V=nd$, $d$ any cardinal,
and subspace $V_1$ of $V$ of $\dim V_1=d$  there is a
$n$-frame $\bar a$ of $\lt(V)$ such that $ a_1=V_1$.
\item For any $n$-frame $\bar a$ in a modular lattice,
$a_\bot=a_{12}$ implies $a_\bot =a_\top$.
\item
There is a  lattice term $t(x,y,\bar z)$  such that
for any modular lattice $L$  and $4$-frame $\bar a$ in $L$
the following hold:
\begin{enumerate} \item
 $\G(L,\bar a)$ is a group under the multiplication
$(g,h) \mapsto t(g,h,\bar a)$ and with   neutral element
$a_{12}$.
\item 
 If $V$ is a vector space and
$L=\lt(V)$,  
then there is a unique isomorphism $\vep_{\bar a}:a_1 \to a_2$
such that $\Gamma_{\bar a}(f):=\{v- \vep_{\bar a}(f(v))\mid v \in a_1\}$
defines an isomorphism of $\GL(a_1)$ onto
$\G(L,\bar a)$.
\item In (b), the subgroup $G$ 
generated by $f_1, \ldots ,f_k$ in $\GL(a_1)$ acts
 fixed point free on $V_1=a_1$
if and only if $a_{12}\cap \bigcap_{i=1}^kg_i=a_\bot$
where $g_i= \Gamma_{\bar a}(f_i)$.
\end{enumerate}
\item
With any conjunction $\pi(\bar x)$ of group equations one
can effectively associate a conjunction 
$\pi^\#(\bar x, \bar z)$ of lattice equations such
that for any modular lattice $L$ and $\bar g=(g_1, \ldots ,g_k)$ and
$\bar a$ in $L$
one has $L\models \pi^\#(\bar g,\bar a)$ 
if and only if $\bar a$ is a $4$-frame in $L$,
 $\bar g$ in $\G(L,\bar a)$,  $\G(L,\bar a)\models \pi(\bar g)$,
and $a_{12}\cap \bigcap_{i=1}^kg_i=a_\bot$.
\end{enumerate} 
\end{lemma}

\begin{proof} 
(i) We may assume $V=\bigoplus_{i=1}^n V_i$
with isomorphisms $\vep_j:V_1 \to V_j$ for $j>1$.
Put $a_i=V_i$, $a_{1j}=\{ v -\vep_j(v)\mid v \in V_1\}$ 
and $a_{kj}= (a_k+a_j)\cap  (a_{1k}+a_{1j})$ for $j\neq k$ in
$\{2,\ldots ,n\}$.

(ii)
  Given a $n$-frame $\bar a$ in $L$ we may assume
$a_\bot=0$ and $a_\top =1$.  
For readability, we write meets as $a \cap b=ab$.
Now, if $a_{12}=0$, then $a_1=a_1(a_{12}+a_2) = a_1a_2=0$ 
and then $a_j=a_j( a_1+a_{1j})=a_j a_{1j}=0$ for all $j>1$. Thus
$a_\top =0$.

(iii) 
We deal with an arbitrary $4$-frame $\bar a$, uniformly,
so that it is obvious which terms govern the construction.
Again, we may assume $a_\bot=0$ and $a_\top=1$.
Let $a_i'=\sum_{j\neq i} a_j$ 
and  observe that, for $i \neq k$,  $a_{ik}\oplus a_i' =1$.
By modularity,  $x \mapsto x+a_{ik}$ and $y \mapsto ya_i'$
are isomorphisms $[0,a_k'] \to [a_{ik},1]$ and  $[a_{ik},1]  \to [0,a_i']$
between intervals 
and compose to the isomorphism
  $\pi^i_k:[0,a_k'] \to [0,a_i']$, that is $\pi^i_k(x)=(x+a_{ik})a_i'$,  
 with inverse $\pi_i^k$.    Observe that $\pi^i_k(a_i)= a_k$
and $\pi^i_k(a_{ij})= a_{kj}$ for $j\neq i,k$.
 Moreover, $\pi^i_k$ is identity on $[0, \sum_{j\neq i,k} a_j]$;
indeed, by modularity, 
 $(x+a_{ik}) a_i' = x+ a_{ik} a_i' =x$ if $x \leq a_i'$.    

Now, let $G_{ij}=\{ x \in L\mid a_i \oplus x= a_j \oplus x = a_i+a_j\}$ 
for $i\neq j$ and observe that $\pi^i_k$ restricts to
a bijection  $\pi^{ij}_{kj}:G_{ij}\to G_{kj}$  for $k\neq i,j$.
Observe that for $x \in G_{ij}$ one has
$x a'_i= x (a_i+a_j)a'_i =x a_j=0$ and, similarly, $x a_j'=0$.

 For $r \in G_{12}$ define
$r_{12}=r$, $r_{1j}= \pi^2_j(r)\in G_{1j}$, and $r_{ij}= \pi^1_i(r_{1j})\in G_{ij}$ 
where $1<i<j$.  Observe that,
for $r =a_{12}$,  this is consistent
with the notation for the $a_{ij}$.
Given  $r,s\in G_{12}$ one has 
$x=(r_{12}+s_{23})(a_1+a_3) \in G_{13}$. Namely, by modularity,
$x a_1  =[r_{12}+s_{23}(a_1+a_2)]a_1
= r_{12} a_1=0$,
$x+ a_1= (r_{12}+ a_1 +s_{23})(a_1+a_3)= (a_1+a_2+s_{23})(a_1+a_3) 
    = (a_1+a_2+a_3)(a_1+a_3)= a_1+a_3$, and,
similarly,   
$x a_3  =0$ and $x+a_3= a_1+a_3$.  
Thus, one obtains a well defined multiplication on $G_{12}$
such that 
\[ r \otimes s =t \quad \Leftrightarrow \quad (r_{12}+s_{23})(a_1+a_3)
=t_{13}.\]
$a_{12}$ is a right unit
since  $(r_{12}+a_{23})(a_1+a_3)=\pi^2_3(r)=r_{13}$.     
 Given  $r \in G_{12}$ \
has image 
$s= \pi^3_2 \pi^2_1 \pi^1_3(r)$ in  $G_{12}$  and,
applying the inverse maps,  it follows
$s_{23}= \pi^1_2 \pi^2_3(s)  = \pi^1_3(r)$.    
Thus, by modularity, $(r_{12}+s_{23})(a_1+a_3)= 
[r+ (r+a_{13})(a_2+a_3)](a_1+a_3) 
= (r+a_{13}) (r+a_2+a_3)(a_1+a_3) 
= [r(a_1+a_3) +a_{13}](a_1+a_2+a_3) 
= a_{13}$ which shows that $s$ is a right inverse of $r$.  
        
In order to establish $G_{12}$ as a group it remains to prove associativity.
Preparing for this, we show for pairwise distinct $i,j,k,h$
\[  \pi^{kj}_{hj}\pi^{ij}_{kj} =\pi^{ij}_{hj},\quad
\pi^{jk}_{hk}\pi^{ij}_{kj}=\pi^{ih}_{kh}\pi^{ji}_{hi}. \]
Indeed, for $x \in G_{ij}$, modularity yields
$[(x+a_{ik})(a_j+a_k) +a_{kh}](a_j+a_h)
= [(x+a_{ik})(a_j+a_k+a_h) +a_{kh}](a_j+a_h)
=(x+a_{ik} +a_{kh})(a_j+a_h)
=[x+(a_{ik} +a_{kh})(a_j+a_i+a_h)](a_j+a_h)
=(x+ a_{ih})(a_i+a_h)$ and 
$[(x+ a_{ik})(a_j+a_k) + a_{jh}](a_h+a_k) 
=[(x+ a_{ik})(a_j+a_k+a_h) + a_{jh}](a_h+a_k) 
=(x+ a_{ik} + a_{jh})(a_h+a_k)=\ldots =
 [(x+ a_{jh})(a_i+a_h) + a_{ik}](a_h+a_k)$.

We now claim that for $r,s,t \in G_{12}$ the following
relations  are equivalent
\[\begin{array}{l}
(r_{12}+s_{23})(a_1+a_3)=t_{13}\\
(r_{12}+s_{24})(a_1+a_4)=t_{14}\\
 (r_{13}+s_{34})(a_1+a_4)=t_{14}\\
(r_{23}+s_{34})(a_2+a_4)=t_{24}
\end{array} \]
whence each equivalent to $t=r \otimes s$.
Namely, the pairs of consecutive   relations are  
equivalent via the isomorphisms $\pi^3_4$, $\pi^2_3$, and $\pi^1_2$,
respectively, which match the elements associated with $r$
and, similarly, for  $s$ and $t$. 
Indeed, 
$\pi^3_4 (s_{23})= \pi^3_4\pi^1_2 \pi^2_3(s_{12})
= \pi^1_2 \pi^3_4 \pi^2_3(s_{12})=
\pi^1_2  \pi^2_4(s_{12})= s_{24}$,
$\pi^3_4(t_{13})= \pi^3_4\pi^2_3(t_{12})= \pi^2_4(t_{12})=
t_{14}$,   $\pi^2_3(s_{24}) =\pi^2_3 \pi^1_2 \pi^2_4(s_{12})
= \pi^1_3 \pi^2_4(s_{12}) =s_{34}$, while
the remaining equalities are obvious.

Now, for $r,s,t \in G_{12}$,
it follows by modularity  $[(r\otimes s)\otimes t]_{14}
=[(r \otimes s)_{13} +t_{34}](a_1+a_4) 
= [(r_{12}+s_{23})(a_1+a_3)+t_{34}](a_1+a_4) 
= [(r_{12}+s_{23})(a_1+a_3+a_4)+t_{34}](a_1+a_4)
= (r_{12}+s_{23}+t_{34})(a_1+a_4)
= [(r_{12}+(s_{23}+t_{34})(a_1+a_2+a_4)](a_1+a_4)
= [(r_{12}+(s\otimes t)_{24}](a_1+a_4)
=[r\otimes (s \otimes t)]_{14}.$\\

Observe that $G_{12}=\G(L,\bar a)$ as sets.
We turn $\G(L,\bar a)$ into the group opposite to $G_{12}$
defining  multiplication as $(g,h)\mapsto t(g,h,\bar a)$  
where $t(x,y,\bar z)$ is the term\\
$\bigl([y+((x+z_{23})(z_1+z_3)+z_{12})(z_2+z_3)](z_1+z_3) +z_{23}\bigr)(z_1+z_2)$.\\

In (b) 
observe that for any $i\neq j$
there is a $1$-$1$-correspondence between
elements  $r \in G_{ij}$ and isomorphisms $f:a_i \to a_j$
given by $f(v)=w$ if and only if $v-w \in r$. 
We write $f=\hat{r}$.
For $s\in G_{jk}$ and $t=(r+s)\cap (a_i+a_k)$ 
it follows $\hat{t}= \hat{s} \circ \hat{r}$. 
In particular we have the
 $\hat{a}_{ij}:a_i \to a_j$ with  $\hat{a}_{ji}=\hat{a}_{ij}^{-1}$ 
and $\hat{a}_{jk} \circ \hat{a}_{ij} =\hat{a}_{ik}$.
We have to put $\vep_{\bar a}=\hat{a}_{12}$.
 Now, given $g,h \in \GL(a_1)$
and $r=\Gamma_ {\bar a}(g)$, $s=\Gamma_ {\bar a}(h)$
one has 
$\hat{r}_{12}= \hat{a}_{12} \circ g$ and
$\hat{s}_{23}=  \hat{a}_{23}\circ \hat{a}_{12}\circ h  \circ \hat{a}_{21}
= \hat{a}_{13}\circ h  \circ \hat{a}_{21}$
whence , for $t= r \otimes s$,
$\hat{t}_{13}= \hat{s}_{23} \circ \hat{r}_{12} 
= \hat{a}_{13}\circ h \circ g$ and $\hat{t}= \hat{a}_{31} \circ \hat{t}_{13}
= h \circ g$.\\

(c)
 Define
\[U=a_{12} \cap \bigcap_{i=1}^k g_i
=   \{ v -\vep_{\bar a} (v) \mid v \in V_1\}
\cap  \bigcap_{i=1}^k \{ v -\vep_{\bar  a}(f_i v)\mid v \in V_1\}.\]
Consider $v$ such that  $v -\vep_{\bar a} v  \in U$.
For  every generator $f_i$ of $G$ there is
   some $w\in V_1$  such that $v -\vep_{\bar a}(v)=
w -\vep_{\bar a}(f_i w)$ whence
$w=v$ and
 $v=f_iv$ since the sum $a_1\oplus a_2$ is direct.
Thus, $v$ is fixed under the action of $G$ on $V_1$.
Conversely, if $v=f_iv$ for all $i$ then
$v -\vep_{\bar a} v  \in U$.
Now, observe that $U=0$ 
if and only if $\{ v \in V_1\mid v -\vep_{\bar a}(v) \in U\}=0$.

(iv) is  obvious by (iii).  
\end{proof}

\section{Consistency in modular lattices}  \lab{MOLs}

\begin{lemma}\lab{5}
There is a recursive set $\Sigma$ of conjunctions
of lattice equations such that the following hold
for any $\phi \in \Sigma$.
 
\begin{enumerate}

\item Given a division ring $F$ and a cardinal $\kappa\geq \aleph_0$.
 If $\phi$ admits a
non-trivial  satisfying assignment in some  modular lattice, then
it does so  within 
 $\lt(V)$
for some $V \in \mc{V}_F$ with $\dim V=\kappa $.
\item 
If $\phi$ admits a
non-trivial  satisfying assignment in the finite modular lattice $L$
and if $F$ is a division ring of  $\chi(F)=0$ or 
$\chi(F)>|L|$ $($it suffices to require
$\chi(F)> |\G(L,\bar a)|$ 
for all $4$-frames $\bar a$ in $L${}$)$ 
 then
$\phi$ admits a
non-trivial  satisfying assignment
 within $\lt(V)$ for some $V\in \mc{V}_F$
with    $\dim V =4d<\aleph_0$ for some $d$.
\item 
$\phi$ is of the   form
$\phi(\bar x,\bar z)$, with $\bar z$ referring to $4$-frames.
$\phi^\exists$ given by
$\exists \bar z \exists \bar x.\; \phi(\bar x,\bar z)
\wedge z_\bot\neq z_\top$
is valid 
within the  modular lattice $L$ if and only if
$\phi$ admits a non-trivial satisfying assignment within $L$.

\item The sets of all $\phi \in \Sigma$ 
with $\phi^\exists$ valid in some, respectively, some finite modular
lattice are not recursive.
\end{enumerate}
The statements remain valid  with  constants $0=z_\bot$
and $1=z_\top$.
 \end{lemma}

\begin{proof}
Let $\Sigma$ consist of all
$\pi^\#(\bar x,\bar z)$, 
according to Lemma \ref{fex}(iv),
with $\pi(\bar x)$ a conjunction of group equations.
We claim:
\begin{itemize}
\item[(*)] $\pi(\bar x)$ admits a non-trivial satisfying  assignment
within some (finite) group $G$ if and only if
$\pi^\#(\bar x,\bar z)$ does so within some (finite) modular lattice $L$;
moreover, given  a non-trivial satisfying  assignment
within $G$, 
  $L$ can be chosen as $\lt(V)$  as in (i) respectively (ii).
\end{itemize}
Clearly,  for  a  
    modular lattice $L$,  $(\bar g,\bar a)$ 
is a satisfying assignment within 
$L$
if and only if $\bar g$
 is a satisfying assignment
for $\pi(\bar x)$ in the   group $G=\G(L,\bar a)$ 
 --- which is finite if $L$ is finite.

If the assignment $\bar g$ in $G$ is trivial, then 
$g_i=e_G=a_{12}$ for all $i$. 
 On the other hand, $a_{12}\cap \bigcap_i g_i=a_\bot$
whence $a_{12}=a_\bot$ and the assignment $(\bar g,\bar a)$ is trivial in view of 
Lemma \ref{fex}(ii). 
In other words, if $\pi(\bar x)$ admits only trivial
satisfying assignments within (finite) groups,
then  $\pi^\#(\bar x,\bar z)$ does so within (finite) modular lattices.

Now,  assume that $\pi(\bar x)$ 
has a non-trivial assignment $\bar h$ in some
(finite) group $G$: in particular $G$ is not the trivial group.
 We have to find assignments in suitable $\lt(V)$.
 We may assume that $G$ is
at most countable. If $G$ is finite, in (i)  we  may 
consider some countably infinite extension, instead. 
 Let $\kappa \geq\aleph_0$ resp. $\kappa=4(|G|-1)$ if $G$ is finite
and choose $V$ of $\dim V=\kappa$ as required in (i) 
and (ii), respectively. 
By Lemma \ref{fex}(i), there is a $4$-frame $\bar a$ of $L=\lt(V)$   
such that $\dim V_1 =\kappa$ for $V_1=a_1$. Fact
\ref{gp} provides a fixed point free faithful
representation $\rho$ of $G$ in $V_1$;
that is, the $f_i=\rho(h_i)$
generate a  subgroup of $\GL(V_1)$ acting fixed point free on $V_1$.
Recall Lemma~\ref{fex}(iii)--(iv) and let $g_i=\Gamma_{\bar a}(f_i)$
to obtain 
 a non-trivial assignment $(\bar g,\bar a)$ in $L$ such  that 
 $L\models \pi^\#(\bar h, \bar a)$.

This proves (*), (i), and (ii).
(iii) is obvious and (iv)
 follows 
from (*),  Theorem \ref{p0},
and the reduction $\pi \mapsto \pi^\#$.
 \end{proof}

In view of these Facts, we consider the following \emph{richness conditions}
on a class $\mc{C}$  of lattices respectively $\mc{V}$ of vector spaces.
\begin{itemize}
\item[(I)]  $\lt(V) \in \Q\mc{C}$ for some vector space $V$
of $\dim V\geq\aleph_0$.
\item[(II$_n$)] For every  $0<d<\aleph_0$ there are 
a division ring $F$ of characteristic not dividing
$d+1$ and  an $F$-vector space $V$ of $\dim V=nd$ such that
 $\lt(V)\in \Q\mc{C}$.
\item[(III$_n$)] $\mc{V}$ consists of finite dimensional
vector spaces over division rings which are finite
dimensional over the center;
moreover, for any $0<d < \aleph_0$
there is $V \in \mc{V}$
of characteristic not dividing $d+1$  such that $\dim V=nd$.
\end{itemize}
Clearly, (III$_n)$ implies (II$_n$) for
$\lt(\mc{V}):=\{\lt(V)\mid
 V\in \mc{V}\}$.
We refer of (II$_4)$ and (III$_4)$
just as (II) and (III); except for Section 9, these 
are the ones to be used.

\begin{theorem}\lab{5c}
The
consistency problem is unsolvable for any
 class $\mc{C}$ of modular lattices
(with or without bounds) 
 satisfying (I) 
or satisfying
 $\mc{C}\subseteq \Q\mc{M}_f$
and (II).
In the bounded  case, unsolvability also persists
in any expansion of $\mc{C}$.\end{theorem}

In particular, in Corollary \ref{lip3}, below,
we  will show that case (II)  applies to
$\mc{C}=\lt(\mc{V})$ where $\mc{V}$ satisfies (III).

\begin{proof}
In view of Lemma \ref{5} and the richness condition
a conjunction of lattice equations admits
a non-trivial satisfying assignment within $\mc{C}$
if and only if it does so within some (finite)
modular lattice --- in case (II) use Fact \ref{tri}.
 \end{proof}

We conclude the section discussing restricted variants of the
consistency problem for modular lattices. These
are not needed for the applications to other structures.

\begin{corollary}\lab{five}
The decision problems of  Theorem \ref{5c}  
remain unsolvable if restricted to
conjunctions $\pi(x_1, \ldots ,x_5)$ of  equations.
\end{corollary}
\begin{proof}
Recall from \cite{quad} 
that the modular lattice  freely generated by a $(k+1)$-frame
is  finitely presented as a modular lattice
with four generators, the frame given by a system
$\bar b(\bar y)$ of terms,  $\bar y=(y_1,y_2,y_3,y_4)$,
 and finitely many  relations. 
Dealing with a conjunction of group equations in $k$ variables $\bar x$, 
encode these adding to $\bar y$ a single lattice variable $y_5$ 
and finitely many relations. Namely,
considering a $(k+1)$-frame $\bar b$ in a modular lattice $L$, 
let the $4$-frame  $\bar a$  given by the $b_i,b_{1j}$,
$i,j \leq 4$ and $L'=[0,\sum_i a_i]$. Then the $x_i$
correspond to  $g_1, \ldots ,g_k$ in $\G(L',\bar a)$. Let    $g'_1=g_1$ and 
$g'_i= (b_1+b_{i+1})\cap (b_{2,i+1}+g_i)$ for $i>1$.
Then $g_i= (b_1+b_{i+1})\cap (b_{2,i+1}+g'_i)$ for $i>1$ 
and  $g_i'= (b_1+b_{i+1})\cap c$ where
$c=  \sum_{i=1}^{k} g_i'$. Introducing the variable $y_5$ 
for $c$ and the associated equations,
  this yields  the conjunction $\psi$ of
$5$-variable lattice  equations  
replacing  $ \pi^\#$ from Lemma \ref{fex}(iv).  
In view of Lemma~\ref{fex}(ii), $a_\bot=a_\top$ implies 
$b_\bot=b_\top$.
 \end{proof}

For a  field $F$ and $\mc{V}=\{F^d_F\mid d<\aleph_0\}$,
 if satisfiability of conjunctions of ring equations 
is decidable for $F$, then the reasoning of \cite[Theorem 4.10]{jacm}
shows that the  consistency problem for $\lt(\mc{V})$ is
solvable if and only if 
 there is a  recursive function $\delta$
that for every  conjunction 
   $\psi$ of lattice equations  one has the following:  
If $\psi$ is  
of binary length $n$ and  satisfiable in $\lt(F^d_F)$ for some $d$
then $\psi$ is also satisfiable in  
$\lt(F^d_F)$ for some  $d \leq \delta(n)$.
By  Theorem \ref{5c}, no such $\delta$ exists if
$F$ is the field of real or complex numbers.

 On the other hand,
in the presence of an orthocomplementation,
Example 4.2(b) in \cite{jacm} gives a recursively defined sequence
$t_k(\bar x)$ of terms  of length $\mc{O}(k)$ in $2k+1$ variables 
such that $t_k(\bar x)=\One$ is satisfiable in  $\lt(F^d_F)$  if $d=2^k$ but not for $d<2^k$.
We provide an analogous recursive sequence   without orthocomplementation
and with fixed number of variables. 

In \cite{group} the \emph{bit length}
of a group presentation is defined as the total number of bits
required  to write the presentation;
in particular, words are
considered  as  strings of powers of generators
and inverses of generators, the exponents 
encoded in binary.  Transferring this to lattice presentations,
we allow the use of recursively defined subterms, encoding
the number of iteration steps in binary.

\begin{corollary}\lab{fast}  There
is a recursively defined sequence 
of  conjunctions $\psi_n(\bar y)$, $n >7$, 
of bounded lattice equations 
 in $5$ variables  $\bar y$
such that $\psi_n$ is  of bit length $\mc{O}(\log n)$
and such that, for any field of $F$ of characteristic $0$,
$\psi_n(\bar y)$  is satisfiable in some $\lt(V_F)$,
with $\dim V_F=d>0$  for $d=4(n-1)$ but
not for $d<4(n-1)$.
\end{corollary}
\begin{proof}
By \cite[Theorem C]{group}
the alternating  groups $A_n$, $n >7$,
have presentations of
bit length $ \mc{O}(\log n)$ in  $3$ generators $\bar x=(x_1,x_2,x_3)$;
and any non-trivial  irreducible  representation of $A_n$ has degree 
$\geq n-1$ \cite{wiman}. 
Put $z_\bot=0$, $z_\top=1$ and 
define  $\psi_n$
as $\pi_n^\#(\bar x,\bar z)$
associated with such  presentation of $A_n$
according to Lemma \ref{fex}(iv). 
 By  Lemma \ref{fex}(iii)(a) there is a constant $K$
such that  for every group word $w(\bar x)$ 
one has a  lattice term $w_{\bar z}(\bar x)$ 
(in the extended sense)  
 such that $|w_{\bar z}(\bar x)|
\leq K|w(\bar x)|$
and $w_{\bar a}(\bar x)$ evaluates as $w(\bar x)$  in any $\G(L,\bar a)$.
Use the proof of Corollary \ref{five} to replace the  
$13$ variables $\bar x,\bar z$ by $5$ new ones. 

Now recall Lemma \ref{fex}(iii) and observe that for any $4$-frame $\bar a$ of $\lt(V)$
and subgroup $G$ of $\GL(a_1)$  one has a $G$-invariant subspace 
$U_1=\{ v\in a_1\mid v -\vep_{\bar a}(v) \in U\}$ of $V_1=a_1$ 
where $U$ is as in the proof of Lemma \ref{5}.
Now,  $U_1=0$ if and only if $U=0$.
Thus, any non-trivial irreducible representation of $A_n$ in some $V_1$ gives
rise to a non-trivial satisfying assignment for $\psi_n$
in $L(V)$, $V=V_1^4$.  
Conversely, any non-trivial satisfying assignment $\bar g, \bar a$
 for $\psi_n$ in some $L(V)$, $V$ a finite dimensional $F$-vector space,
we may assume $a_\bot=0$ and $a_\top=V$ and
$\bar g$ defines a non-trivial representation
of $A_n$ in $V_1=a_1$ which, by Maschke's Theorem,
has a non-trivial direct summand, whence $\dim V \geq 4(n-1)$. 
\end{proof}

\section{Relation Algebras} \lab{RelationAlgebras}

A \emph{pre-relation algebra}
is an algebraic structure $A$ with two binary operations
written as $\cap$ and $\circ$, a unary operation $^{-1}$,
 and constant $\Delta$.
We write $\alpha \in \Eq(A)$ if 
$\Delta \cap \alpha =\Delta$, $\alpha^{-1}=\alpha$, and
$\alpha \circ \alpha =\alpha$.
We also consider the partial algebra
 $A^\#$ where $\circ$ is replaced by the partial
operation given by
$\alpha+ \beta = \gamma$  if and only if $\alpha, \beta \in \Eq(A)$
and $\alpha \circ \beta =\beta\circ \alpha =\gamma$.
We write $\alpha\oplus \beta =\gamma$ if   
$\alpha+ \beta =\gamma$ and $\alpha \cap \beta =\Delta$.
A system $\bar \alpha$  in $\Eq(A)$ is a \emph{permuting $4$-frame}
of $A$ 
if the  equations defining a $4$-frame in  a lattice
are satisfied by $\bar \alpha$, being  evaluated   within $A^\#$,
 and if $\alpha_\bot=\Delta$.
Define
 \[\G(A,\alpha)=\{\beta \in \Eq(A)\mid \beta\oplus \alpha_1=
\beta\oplus \alpha_2=\alpha_1\oplus\alpha_2\} \]
Given a set $S$ we consider  
the pre-relation algebras  on  sets of  binary relations 
on $S$ with the following operations: intersection $\cap$,
relational product $\circ$,  inversion $^{-1}$,
and $\Delta=\id_S$. We say that  $A$ is  \emph{represented} on $S$
- J\'{o}nsson \cite{jon}  calls $A$ an \emph{algebra of relations}. 
Let $\mc{R}$  denote the class of all
 algebras isomorphic to such ---
the class  of \emph{representable}
pre-relation algebras; $\mc{R}$ is quasi-variety by 
\cite[Theorem 1]{jon}. By $\mc{R}_f$ we denote the class of 
finite members of $\mc{R}$. 
The following are  immediate by \cite[Corollary 2]{hcom} 
and Lemma \ref{fex}(i).

\begin{fact}\lab{8}
If $\alpha$ is a permuting $4$-frame of
 $A$ which is represented on $S$,
then the subalgebra $L=\lt(A,\alpha)$ generated by $\alpha$ together
with $\G(A,\alpha)$ consists of pairwise 
permuting equivalence relations on $S$;
with operations $\cap$ and $+$ it forms 
a modular sublattice ${\sf L}(A,\alpha)$  of $\Eq(S)$;
in particular, $L\subseteq \Eq(A)$ and $L^\#=L$.
Moreover, $\alpha$ is a $4$-frame of $L$
and $\G(A,\alpha)=\G(L,\alpha)$. 
\end{fact}

\begin{fact}\lab{10}
For any vector space  $V$ of $\dim V =4d$   
there is a pre-relation algebra $A=\A(V)$ 
represented on $V$ 
and a permuting $4$-frame
of $A$  
such that  $\lt(V)={\sf L}(A,\alpha)$. 
\end{fact}

We consider particular formulas  in the language
of pre-relation algebras:  $\Eq(\bar y)$
is the conjunction of   equations
 such that $A\models \Eq(\bar \xi)$ 
implies $\beta \in \Eq(A)$ for any $\beta$ in the list
$\bar \xi$. 
A \emph{type-1-formula} $\psi(\bar y)$
is a conjunction of basic  equations   in the $\cap$-$\circ$-fragment.
Let $\tau(u,\bar y, \bar v)$ denote the obvious  type-1-formula such that
$A\models \tau(\gamma,\bar \delta,\bar \vep)$
for some $\bar \vep$ 
if and only if $\gamma =\delta_1 \circ \ldots \circ \delta_n$
and that, in this case,   $\gamma=\vep_j=\delta$ for all $j$ if
$\delta_i =\delta$ for all $i$.
The richness conditions (I) and (II) are  modified
replacing $\lt(V)$  by $\A(V)$.

\begin{theorem}\lab{11}
The consistency problem is unsolvable for any
class  $\mc{A}$ of representable pre-relation algebras
satisfying (I)
or    contained in $\Q\mc{R}_f$ 
 and satisfying (II).
More precisely, there is a recursive set $\Sigma'$
of type-1-formulas  such 
that there is no algorithm which on input $\psi(\bar y)\in \Sigma'$
would decide whether $\psi^\exists$ given as  \[
\exists \bar y \exists u \exists \bar v.\; 
\Eq(\bar y) \wedge  \bigcap_i y_i  =\Delta \wedge
\psi(\bar y) \wedge \tau(u,\bar y, \bar v) \wedge u \neq \Delta\]
is satisfied in some member of $\mc{A}$. 
\end{theorem}

\begin{corollary}\lab{11c}
Unsolvability persists 
if the total relation is given as constant $\nabla$;
in this case, $ u=\nabla$ has to be added as a 
conjunct in forming   $\psi^\exists$.
Also unsolvability persists in any expansion of $\mc{A}$.
\end{corollary}

\begin{proof}[Proof of Theorem \ref{11}]
Recall $\Sigma$ from Lemma \ref{5};
to obtain $\Sigma'$
 replace $\phi(\bar x,\bar z)$ by an equivalent
unnested pp-formula
in lattice language (Fact \ref{p1}). In the latter, replace any equation
$u=v+w$ by $u=v \circ w\; \wedge \;u=w \circ v$.
This  associates with $\phi(\bar x,\bar z)$
 a type-1-formula $\phi'(\bar x,\bar z,\bar y)$
such that 
 in any  lattice $L$ of permuting equivalences
one has $(L,\cap,+)\models \phi(\bar \xi,\bar \alpha)$
if and only if $(L,\cap,\circ)\models \exists \bar y.\;
\phi'(\bar \xi,\bar \alpha, \bar y)$.
Observe that, if ${\phi'}^\exists$ is valid in
$A \in \mc{R}$ then $\phi^\exists$ is valid
in  ${\sf L}(A,\alpha)$ 
for some witnessing permuting $4$-frame $\alpha$ in $A$.
Conversely, if $\phi^\exists$ is valid in 
$L={\sf L}(A,\alpha)$, then  ${\phi'}^\exists$
is valid in $A\in \mc{R}$ ---
witnessed by  the values for $\bar z$ and $\bar x$
in $L \subseteq A$.
Thus, put $\Sigma'=\{\phi'\mid \phi \in \Sigma\}$.

Now, if $\phi^\exists$ is valid within some (finite) modular
lattice, then it is so within suitable $\lt(V)$
(Lemma \ref{5}(i),(ii)) whence ${\phi'}^\exists$ is valid
 in $\A(V) \in \mc{A}$ (cf. Fact~\ref{10}).
Conversely, if ${\phi'}^\exists$ is valid in $A\in \mc{A}$
(by Fact \ref{tri} we may assume
$A$ finite in case (II)), then 
$\phi^\exists$ is valid in the (finite) modular
lattice ${\sf L}(A,\alpha)$ where $\bar \alpha$ 
is   witness   for $\bar z$.
Thus, the claim follows from  Lemma \ref{5}(iii),(iv) and
the reduction $\phi \mapsto \phi'$.
\end{proof}

\section{Databases} \lab{databases}
We follow  \cite{kan} for database concepts,
though adapting notation to common use
in mathematics.
  Fix a countably
infinite  set $X_{\infty}$ of variables and use  $x,y, \ldots$
to denote elements of $X_{\infty}$.
Under the pure universal relation assumption, a \emph{database} $D$ is given
by a finite non-empty $U \subseteq X_{\infty}$
of \emph{attributes}, for each $x$ in $U$ a
 domain $\Delta[x]$ of \emph{values} of the attribute  $x$,
and a non-empty subset (\emph{relation}) $R$ of the direct product $\prod_{x \in U} \Delta[x]$.
For a tuple $t$ in $R$ and
 $X \subseteq U$ let 
$t[X]$ be the restriction of $t$ to $X$. 
 
 The atomic sentences to be considered are  
the \emph{functional dependencies} (\emph{fd}'s) $X \imp Y$ and
the \emph{embedded multivalued dependencies} (\emph{emvd}'s) $[X,Y]$
 with non-empty finite 
$X,Y \sub U_{\infty}$. 
 Validity is defined as follows:
$D \models X \imp Y$  if and only if for all $ s,t \in R$, if $s[X]=t[X]$ then
$s[Y]=t[Y]$ - provided $X$  and $Y$ are subsets of $U$, at all. 
$D \models [X,Y]$ if and only if
for every $t_1, t_2 \in R$ with $t_1[X \cap Y] = t_2[X \cap Y]$
there exists $t \in R$ with $t[X]=t_1[X]$ and $t[Y]=t_2[Y]$; 
in other words, the restriction of $R$ to $XY=X \cup Y$ is the
natural join of the restrictions of $X$ and $Y$.

 A database $D$
with attribute set $U$  is \emph{trivial} if it satisfies
all fd's and emvd's (with attributes from $U$). 
$D$ is \emph{almost trivial} if each of its 
attributes is a key, that is if $D$ satisfies all
 fd's.

\begin{fact}
A database is almost trivial if $s[x] \neq t[x]$ 
for all $s\neq t$ in $R$ and $x \in U$.
A database is trivial if and only if its attribute set $U$ or its
 relation  $R$ is a singleton
set. 
\end{fact}
\begin{proof}
The first claim is obvious and so is 
the inverse implication in the second.
Now, assume $D$ trivial and $U=X\cup Y$ the disjoint union of non-empty
$X,Y$.
Consider 
$t_1,t_2$ in $R$. 
 In view of the  emvd $[X,Y]$
there must be $t \in R$ such that $t_1[X]=t[X]$ and
$t_2[Y]=t[Y]$. The fd's $x\imp y$
and $y \imp x$ ($x \in X, y \in Y$) imply 
$t_1[Y]=t[Y]$
and  $t_2[X]=t[X]$
whence $t_1=t=t_2$
\end{proof}

For a vector space $V$ and injective map $f:U \to \lt(V)$,
  $U\subseteq X_{\infty}$ finite,
consider the database $\D(V,f)$ where 
 $\Delta[x]$ is given  as  the set of cosets of $f(x)$
and $R= \prod_{x \in U} \Delta[x]$.
Let $\D(V)$ denote the set of all $\D(V,f)$
for given $V$.
The richness conditions (I) and (II) are modified replacing $\lt(V)$ resp. $\A(V)$
by $\D(V)$.

\begin{theorem}\lab{dt}
Consider a class 
 $\mc{D}$  of  databases which either
satisfies (I) 
or which consists of finite databases and
satisfies (II).
Then there is no algorithm which on input of a
conjunction of fd's and emvd's decides
whether there is a member of $\mc{D}$ which is not
almost trivial (resp. non-trivial).
\end{theorem}

While consistency (as a special case of implication)
is decidable for fd's alone, the question for emvd's alone
remains open.
By   \cite[Section 2.2]{juha}
the analogue of the Theorem follows  for   models  in independence 
logic w.r.t. 
inclusion and conditional independence atoms.

\begin{proof}[Proof of Theorem \ref{dt}]
Given a database $D$,
 one has projection maps $\pi_x:R
\imp \Delta[x]$ yielding for each $t \in R$ its 
$x-$component $\pi_{x}(t)=t[x]$ in $\Delta[x]$. 
With each of these maps one has its kernel equivalence relation $\theta_{x}$.
For a set $X$ of attributes write $\theta_X=\bigcap_{x \in X} 
\theta_x$. Thus, $(s,t)\in \theta_X $ if and only if $s[X]=t[X]$.

Let $S$ denote the factor set $R/\theta_U$ 
of equivalence  classes modulo $\theta_U$ 
and, for $X \subseteq U$,  $\eta^D_X$ the equivalence relation on $S$ corresponding to 
$\theta_X$, that is for any $\theta_U$-classes one has
\[ (t/\theta_U, s/\theta_U) \in \eta^D_X \;\mbox{ iff }
(t,s) \in \theta_X. \]
Clearly, $\eta^D_X= \bigcap_{x \in X} \eta^D_x$.
Let $\A(D)$ the pre-relation algebra
represented on $S$ which is generated by the  $\eta^D_x$,  $x \in U$.
It follows  from \cite[Lemma 11]{data}
\[\eta^D_X \subseteq \eta^D_Y \;\mbox{ iff }\; D\models X \imp Y \]
\[\eta^D_Z = \eta^D_X \cap \eta^D_Y
\;\mbox{ iff }\; D \models XY \imp Z \et Z \imp XY \]
and, if $X,Y,Z$ are pairwise disjoint, then
\[\eta^D_Z = \eta^D_X \circ \eta^D_Y 
\;\mbox{ iff }\; D \models [XZ,YZ] \et X \imp Z \et Y \imp Z .\]
For a tuple $\bar y$ from $U$ write $\eta^D(\bar y)
=(\eta^D_{y_1}, \ldots , \eta^D_{y_n} )$.

Conversely, any pre-relation
algebra $A$, represented on a set $S$, together with an assignment 
$\eta$  within $\Eq(A)$ for variables from a finite  $U\subseteq X_\infty$
such that $\bigcap_{x \in U}\eta(x) =\Delta$,
 gives rise to
a database $D=\D(A,\eta))$ with $\Delta[x]$ the set of classes
of $\eta(x)$ and $R$ the image of $S$
under the map $a\mapsto (a/\eta(x)\mid x \in U)$. 
For such, $\theta_U=\Delta$ and  $(A,\eta)\cong (\A(D), \eta^D)$.
Moreover, $D$ is almost trivial if and only if
$\im \eta$ is a singleton set, namely $\{\Delta\}$.

Given a  type-1-formula $\psi(\bar y)$ let $U$
consist of all variables occurring in $\psi^\exists$
and translate $\psi(\bar y) \wedge \tau(u,\bar y,\bar v)$
into 
 a conjunction $\psi'$ of fd's and emvd's
in attributes $\bar y, u, \bar v$
such that
$D \models \psi'$
if and only if \[\A(D)\models\; 
\bigcap_i \eta^D(y_i) =\Delta \;\wedge\; \psi(\eta^D(\bar y))  \;
\wedge\; \tau(\eta^D(u), \eta^D(\bar y), \eta^D(\bar v)).
  \]

Observe that for a database $D=\D(V,f)$
 one has
 $(\A(D),\eta^D)\cong (\A(V),\hat{f})$
where $\hat{f}(x)$ is the equivalence relation associated with $f(x)$.  
Thus,  Theorem \ref{11} applies to
 $\mc{A}=\{\A(D)\mid D \in \mc{D}\}$. We rephrase its  statement: 
There is no algorithm which, on input of a type-1-formula
$\psi(\bar y)$ decides whether  there is $A$ in $\mc{A}$
and an assignment $\eta$ for $\bar y,u,\bar v$ in $\Eq(A)$
such that $A\models \psi(\eta(\bar y))\;\wedge\;\tau(u,\bar y,\bar v) $
and $\bigcap_i \eta(y_i)=\Delta$ but $\eta(u) \neq \Delta$,
the latter being equivalent to $\im \eta$ not to be singleton. 
Deciding the latter  reduces to deciding 
whether there is a non almost trivial  database $D$ with
 attributes $\bar y, u, \bar v$ satisfying 
the conjunction $\psi'$ of fd's and emvd's. 
This shows unsolvability of 
the first  decision problem  in the Theorem.
For the second problem, we have to refer to Corollary \ref{11c}
and translate $u=\nabla$ by  $[u,U\setminus u]\wedge U\setminus u \imp u$.
Indeed, if the latter holds in $D$, then, by the emvd, 
for any $t_1,t_2 \in R$ there is $t\in R$ 
such that $t[u]=t_1[u]$ and $t[U\setminus u]=t_2[U\setminus u]$
 whence, by the fd, $t_2[u]=t[u]=t_1[u]$  and
so $(t_1,t_2) \in \theta_u$.
\end{proof}

\section{Rings} \label{Rings}

We consider rings $R$ with constants $0,1$. 
Let 
$\latr(R)$  denote  the (modular)  lattice of all right ideals.
A ring $R$ 
is (von Neumann) \emph{regular} if
for any $a \in R$ there is $x \in R$ such that $axa=a$;
equivalently,  any of its principal right
ideals is generated by an idempotent.
The principal right ideals of a regular ring  form a
  sublattice $\latn(R)$ of 
$\latr(R)$. The following is well known and easy to prove.

\begin{fact} \lab{endo}
The  endomorphism ring $R={\sf End}(V)$ 
of an vector space  $V$
is regular and one has
$\latn(R) \cong \lt(V)$ via $\varphi R  \mapsto \im \phi$.
\end{fact}

\begin{fact}\lab{regring}
There are positive primitive
$\sigma(x,y,z)$ and $\mu(x,y,z)$
in the language of rings such that the following 
hold  for any idempotents $e,f,g$ in a ring $R$:
\begin{itemize}
\item[(a)] $R\models \sigma(e,f,g)$ if and only if $gR=eR+fR$.
\item[(b)] $R\models \mu(e,f,g)$ implies $gR= eR \cap fR$.
\end{itemize}  
If $R$ is regular, then in (b) holds the converse, too.
\end{fact}
\begin{proof}
Concerning (a) observe that $gR=eR+fR$ if and only 
\[R\models ge=e \wedge gf=f \wedge \exists r\exists s.\,g=er+fs.\]
Concerning (b) observe that, according to 
\cite[Lemma 8-3.12 (ii)]{fred}, if
\[\begin{array}{ll}\exists r \exists s.& (f-ef)r(f-ef)=f-ef\;\wedge\\ 
 &g (f- f r(e-ef))= f- fr(e-ef) \;
\wedge\; g= (f- fr(e-ef))s. \end{array}\] holds in $R$ then 
$gR=eR\cap fR$ --- the first equation encodes that
$r$ is a quasi-inverse of 
$f-ef$ while the last two state that
\[gR=(f- fr(e-ef))R.\]  
 \end{proof}

In the richness conditions 
replace $\lt(V)$ by $\End(V)$. Let $\mc{R}_f$ denote the class of
all finite rings. 

\begin{theorem} \lab{ring}
The consistency problem  is unsolvable
 for any class $\mc{C}$
 of  rings 
 satisfying  (I) or contained in  $\Q\mc{R}_f$ and
satisfying  (II).
\end{theorem} 

\begin{proof} 
Consider the language of bounded lattices and
recall $\Sigma$ from Lemma \ref{5}.
Given $\phi(\bar x) \in \Sigma$, replace it by the equivalent  
 unnested pp-formula $\exists \bar y.\,\phi'(\bar x,\bar y)$ (Fact \ref{p1});
associate with each variable $x,y$ in the latter 
 a ring variable
$\hat{x},\hat{y}$ and let $\chi$ be the conjunction
of all equations  $\hat{x}^2=\hat{x}$ and $\hat{y}^2=\hat{y}$.
Use  Fact~\ref{regring} to replace the  basic
lattice equations  by existentially quantified
conjunctions of ring equations, each  
with new rings variables.
The conjunction of these and of $\chi$ 
is equivalent to 
a positive primitive ring formula $\hat{\phi}(\bar x,\bar y)$
such that $\latr(R)\models \phi^\exists$
if  $R\models \phi^R$ where $\phi^R$ is given as
$ \exists \bar x \exists \bar y.\,\hat{\phi}(\bar x,\bar
y) \wedge 0\neq 1$.
Conversely, if   $\phi^\exists$
holds in $\lt(V)$  then, in view of Fact \ref{endo},
$\phi^R$ holds in $\End(V)$.

In view of  the richness conditions and Fact \ref{tri},
the claim follows from  Lemma \ref{5}(iii),(iv) as in the proof
of Theorem \ref{11}.
\end{proof}

Let  $\mc{N}_f$ denote the class of all finite regular rings.
By the Artin-Wedderburn Theorem  $\mc{N}_f$ 
consists, up to isomorphism, just of the direct
products of matrix rings $F^{d \times d}$,
$d<\aleph_0$, $F$ a finite  field. 

\begin{fact}\lab{lip}
$\End(V) \in \Q\mc{N}_f$
if $V$ is an $F$-vector space of $\dim V <\aleph_0$ and 
 if $F$ is finite dimensional over its center.
\end{fact} 
\begin{proof}
This  can be seen as a variant of Lemma 3.5  in Lipshitz \cite{lip}.
Since $\End V_F$ embeds into $\End V_C$, $C$ the center of $F$,
we may assume that $F$ is  a field
and consider $F^{d \times d} \cong \End(V_F)$.  
By tensoring with $\bar F$, the algebraic closure of $F$,
we have $F^{d \times d} $ embedded into 
$\bar F^{d \times d}$. The  algebraic closure $\bar P$  of the prime subfield
$P$ of $F$ is elementarily  equivalent to $\bar F$,
and it follows that $\bar F^{d\times d}$ and
$\bar P^{d \times d} $ are elementarily equivalent, too.
Now, $\bar P^{d\times d}$
is the  directed union 
(whence in the quasi-variety) of 
the  $K^{d\times d}$ where $K$ is a subfield of $\bar P$ of finite degree,
 --- and  finite if $P$ is finite.
Finally, 
observe that $ \overline{\mathbb{Q}}$  embeds into 
a suitable  ultraproduct of the $\bar P$, $P$ finite,
since that is algebraically closed and of characteristic $0$.
\end{proof}

In particular, if $\mc{F}$ 
is a class of division rings which are finite dimensional over the center
and if $\mc{F}$ contains members of characteristic $0$ or
infinitely many finite characteristics,
then there is no algorithm to decide,
for a given finite family of multi-variate  polynomials $p_i$ (in non-commuting 
variables)  with integer coefficients,
whether 
there is a common zero in
the matrix ring  $F^{d \times d}$ for some $F \in \mc{F}$ and $0<d
<\aleph_0$.
 In view of Fact \ref{p1}
one may restrict to  families of quadratic polynomials.

If the matrix rings  $F^{d \times d}$ are endowed with 
an involution $A \mapsto A^*$ 
such that 
$\sum_i A_iA_i^* =0$ implies $A_i=0$ for all $i$   then a family $(p_i)$
can be replaced by the single  $\sum_i p_i p^*_i$,
 which can be considered a polynomial in variables $x_i, x^*_i$,
to be interpreted such that $x^*_j \mapsto B^*_j$ 
if $x_j \mapsto B_j$.  Again, it suffices to consider
a single quartic such polynomial. In particular this
applies if $\mc{F}$ consists of  subfields of the
 complex numbers, closed under conjugation, and if
$A^*$ is  the conjugate transpose of $A$.

In the context of the categorical approach to
Quantum Theory (cf. \cite{abr,hard}), Theorem \ref{ring}
yields the following: Let $F$ be a division ring of
characteristic $0$ and 
$C$  the additive category, possibly
enriched  with additional structure, of 
 finite dimensional $F$-vector spaces.
Consider $C$ as a partial algebraic structure
the underlying ``set'' of which  is the class of all morphisms.
Then
there is no algorithm to decide, for any given
conjunction $\pi(\bar x)$ of equations, whether 
$\pi(\bar x)$ admits an
assignment in $C$ which is satisfying
(in a particular, having all terms in $\pi(\bar x)$ evaluated)
and non-trivial (that is, not having a $0$-morphism as single  value).
Indeed, the problem of Theorem \ref{ring} can be encoded
so that satisfying assignments must have  values 
which are endomorphism of a single object.

\section{Complemented modular lattices} \lab{CMOLs}

A modular lattice $L$ with bounds $0,1$
as constants  is \emph{complemented} 
if for any $a$ there is $b$ such that $a \oplus b=1$
(in the sequel, we consider $0,1$ as constants).
Here, 
 consistency problems can be given a more special  form.

\begin{fact}\lab{lat}
 Within the class of  complemented modular lattices, any conjunction  of equations is equivalent to a formula
 $\exists \bar y.\;s(\bar x,\bar y)=0\;\wedge\; t(\bar x,\bar y)=1$
with terms $s,t$.
\end{fact}

\begin{proof} 
 Given a  conjunction of equations $s_j=t_j$,
observe each $s_j=t_j$ equivalent 
 to 
$\exists v: \tilde s_j=\Zero \;\wedge\; \tilde t_j=\One$
for  $\tilde s_j:=(s_j+ t_j)\cap v$ and $\tilde t_j:=(s_j\cap
t_j)+ v$
(due to modularity and existence of complements);
and $\tilde s_j=\Zero \;\wedge\; \tilde t_j=\One\;\wedge\;\tilde s_i=\Zero \;\wedge\; \tilde t_i=\One $
equivalent to $\tilde s_j+\tilde s_i=\Zero\;\wedge\;\tilde
t_j\cap\tilde t_i=\One$. 
\end{proof}

In particular, the lattices $\lt(V)$ of all linear subspaces
of vector spaces are complemented modular
and so are the lattices $\latn(R)$
of principal right ideals of regular rings.
For the latter, the following is useful in case (II).

\begin{fact}\lab{qua}
For any class $\mc{R}$ of regular rings
one has $\Q\{\latn(R)\mid R \in \mc{R}\}  \subseteq 
\{\latn(R)\mid R \in \Q\mc{R}\}$
\end{fact}

\begin{proof}
Since any $R \in \Q\mc{R}$ embeds into some
direct product $P$ of ultrapoducts  $S_i$ of   members $R_{ij}$ of $\mc{R}$
(cf \cite[Corolllary 2.3.4]{gorb})
it suffices to observe that $\latn(R)$ embeds into
$\latn(P)$ via $eR \mapsto eP$ and 
 $\latn(P)$  into  the direct product of the
$\latn(S_i)$ via $(e_iS_i\mid i \in I) \mapsto
(e_i\mid i \in I)P$ (cf. \cite[Corollaries  8-3.14-15]{fred}  
and that the $\latn(S_i)$
 satisfy all quasi-identities
valid in the $\latn(R_{ij})$
  (which, by Fact \ref{regring}
translate into  sentences in the language of 
rings). 
\end{proof}

The following is a lattice theoretic variant
of Fact \ref{lip} and can be proved in the 
same fashion. It follows, immediatedly, if one
combines Facts 
\ref{endo}, \ref{lip}, and \ref{qua}.
Together with Theorem \ref{5c} it implies Corollary 
\ref{lip3}. 
\begin{fact}\lab{lip2}
$\lt(V) \in \Q \mc{M}_f$
if $V$ is an $F$-vector space of $\dim V <\aleph_0$  and 
 if $F$ is finite dimensional over its center.
\end{fact}

\begin{corollary}\lab{lip3} 
Theorem  \ref{5c} and Corollary \ref{five} 
apply to $\mc{C}=\{\lt({V}\mid V \in \mc{V}\}$ where $\mc{V}$ satisfies
the richness condition  (III).
\end{corollary}

Of course, if for a class $\mc{C}$
of complemented modular lattices  choice of some complement is added  as fundamental
operation, Theorem \ref{5c} applies.
If $d=\dim V <\aleph_0$, if
$F$ is a division ring with involution,
and if  $V$ is endowed with an anisotropic form $\Phi$
hermitean with respect to this involution, then 
$U \mapsto U^\perp=\{v \in V\mid \forall u \in U. \Phi(v,u)=0\}$
turns $\lt(V)$ into the  ortholattice $\lt^\perp(V)$.
Here, 
 an \emph{ortholattice} is a bounded lattice
 endowed with 
a  dual automorphism  $x \mapsto x^\perp$ 
of order $2$  such that $x\oplus x^\perp=1$.

In order to have Corollary \ref{lip3} available,
we consider a class $\mc{V}$ of such spaces
where the class of underlying vector spaces
satisfies condition (III). 
Then, by Corollary \ref{five} and Fact~\ref{lat2}, below,  we obtain the following.

\begin{corollary}\lab{hilb}
There is no algorithm which, given a
$5$-variable  term $t(\bar x)$ 
in the language of ortholattices, decides
whether $\exists \bar x.\;t(\bar x)=1$
is valid in the ortholattice $\lt^\perp(V)$ 
for some $V \in \mc{V}$ with $\dim V>0$.
\end{corollary}

Natural examples for $\mc{V}$ 
are the classes of all  finite dimensional real, complex,
and quaternionian, respectively, Hilbert spaces.
Here, in contrast,
deciding whether $\exists \bar x\;t(\bar x)>0$
holds in some $\lt^\perp(V)$ (``weak satisfiability'')
is decidable (cf. \cite{hard2}) and an upper complexity  bound  has been derived
in \cite{jacm}.

\begin{fact}\lab{lat2}
Within the class of   modular ortholattices, any conjunction
$\pi(\bar x)$ of equations 
 is equivalent to an equation
 $t(\bar x)=1$.
\end{fact}

\begin{proof} 
Observe that the following are equivalent for any given $x,y$:
$x+x^\perp y^\perp=\One$; $x^\perp=x^\perp(x+x^\perp y^\perp)$;
 $x^\perp=xx^\perp+x^\perp y^\perp$ (by modularity);
 $x^\perp \leq y^\perp$; $y \leq x$.
Thus, $s_j \leq t_j$ is equivalent to some
$u_j=1$ and $t_j \leq s_j$ to some $v_j=1$;
 and  $\bigwedge_j s_j=t_j$   to $\bigcap_j u_j\cap v_j =1$.   
\end{proof}

\section{Grassmann-Cayley algebra} \label{GrassmanCayley}

Recall, that for a finite dimensional vector space $V$
 the Grassmann-Cayley algebra $\CG(V)$ (cf \cite{sturalgo})
  has, in particular, operations
$\wedge$ and $\vee$  and terms built from them and $0, 1$:
the  \emph{simple expressions}.
These operations are related to the lattice $\lt(V)$ as follows: 
$0,1$ are the bounds of $\lt(V)$,   $A\wedge B=A\cap B$
if $A+B=V$  and $A\vee B=A+B$ if $A \cap B=0$.

\begin{theorem}\lab{grass}
Let $\mc{V}$ be a class of vector spaces which
satisfies (III$_{16})$.
There is no algorithm to decide for any 
given  conjunction of 
equations $t_i(\bar x)=s_i(\bar x)$,
with simple expressions $t_i,s_i$,
whether it admits a satisfying assignment  
within $\CG(V)$ for some $V \in \mc{V}$, $V \neq 0$.
\end{theorem}

The proof needs some preparation.
We consider lattices with bound $0,1$.
For a  term $t(\bar x)$, call the assignment $\bar x \mapsto \bar a$
in $L$ \emph{admissible}  if, for any
occurrence of  subterms $s(\bar x)$,
$s_1(\bar x)$, and $s_2(\bar x)$ in $t(\bar x)$, the following hold
in $L$:
\begin{quote}
If $s(\bar x)= s_1(\bar x)+ s_2(\bar x)$ then 
$s_1(\bar a) \cap  s_2(\bar a)=0$;\\
If $s(\bar x)= s_1(\bar x)\cap  s_2(\bar x)$ 
then $s_1(\bar a) + s_2(\bar a)=1$.
\end{quote}
We say that $\bar a$ is \emph{admissible} for a 
conjunction $\pi(\bar x)$ of  equations if it is so for 
any subterm occurrence in $\pi(\bar x)$. If, in addition, $L\models \pi(\bar a)$ 
then we write $L\models_a \pi(\bar a)$.
For an $n$-frame $\bar a$  of  a modular lattice $L$ and
$i\neq j$ 
put \[R_{ij}(L,\bar a)=\{ x \in L\mid  x\oplus a_j=a_i+a_j\}.\]

\begin{fact}\lab{ad} Fix $n\geq 3$.
\begin{enumerate}
\item There is a conjunction $\phi(\bar z)$ of  lattice
equations such that $\bar a$ is a $4$-frame of the
modular lattice $L$ if and only if $L\models_a \phi_n(\bar a)$.
\item There are lattice terms $\otimes(x,y,\bar z)$ and $\ominus(x,y,\bar z)$
such that for any vector space  $V$ and
$4$-frame $\bar a$ of $\lt(V)$
 there is a (unique) isomorphism $\vep_{\bar a}:a_1 \to a_2$
such that $\Gamma_{\bar a}(f):=\{v- \vep_{\bar a}(f(v))\mid v \in a_1\}$
defines a ring isomorphism of $\End(a_1)$ onto
$R_{12}(L,\bar a)$
 with multiplication $(r,s) \mapsto \otimes(r,s,\bar a)$,
difference $(r,s) \mapsto \ominus(r,s,\bar a)$, zero $a_1$, and unit $a_{12}$.
Moreover, the assignments $x,y,\bar z \mapsto r,s,\bar a$
with $r,s \in R_{12}$ are admissible for the terms 
$\otimes(x,y,\bar z)$ and $\ominus(x,y,\bar z)$.
\item
With any conjunction $\pi(\bar x)$ of ring  equations
of the form $p_i(\bar x)=0$ one
can effectively associate a conjunction 
$\pi^\#(\bar x, \bar z)$ of lattice equations
of the form $t_i(\bar x, \bar z)=z_1$  such
that for any vector space $V$
  and $\bar r,\bar a$ in $L=\lt(V)$
one has $\lt(V)\models_a \pi^\#(\bar r,\bar a)$ 
if and only if $\bar a$ is an $4$-frame of $L$,
 $\bar r$ in $R_{12}(L,\bar a)$, and  $R_{12}(L,\bar a)\models \pi(\bar r)$.
 \end{enumerate}
\end{fact}

\begin{proof}[Proof of Theorem \ref{grass}]
Observe that the class $\mc{C}$ of  $\End(V_1)$ where $V_1^4 \cong V \in \mc{V}$
satisfies   condition (II) for rings while
$\mc{C} \subseteq \Q \mc{R}_f$ follows from Fact \ref{lip}. 
Thus, the  claim follows from the reduction in 
Fact \ref{ad}(iii) and Theorem \ref{ring}
where $\mc{C} =\{\End(V_1)\mid V_1 \in \lt(V), \dim V= 4\dim V_1\}$.
\end{proof}

\begin{proof}[Proof of Fact \ref{ad}]
(i) Consider a $4$-frame of the modular lattice $L$.
Due to the first condition defining a frame,
the assignment $\bar z \mapsto \bar a$ is admissible for 
any term arising from  $\sum_{k\in K} z_k$  by insertion of brackets.
It only remains to deal with the last condition.
Observe that $a_{ij} \cap a_{jk} \leq a_{ij}\cap (a_i + a_j)\cap (a_j +a_k)= a_{ij}\cap a_j=0$, $b + a_{ij} +a_{jk}=1$ where   
$b=\sum_{\ell \neq j}a_\ell$,
 and that the condition is equivalent to
$b\cap (a_{ij}+a_{jk})=a_{ik}$.

(ii)
Recall the approach in the  proof of Lemma \ref{fex}(iii)(a),
in particular the terms $\pi^i_k$. We write $R_{ij}=R_{ij}(L,\bar a)$.
The assignment $x,\bar z \mapsto r,\bar a$ where $r \in R_{ij}$ or
$r \in R_{ji}$ 
is admissible for  $\pi^i_k$ since $r \cap a_{ik} \leq 
(a_i +a_j) \cap a_{ik}=0$ and $\pi^i_k(r)=(r+a_{ik})\cap b$
where $b=\sum_{\ell \neq i} a_{\ell}$ and $b+r+a_{ik}=\One$.
Also, observe that the map $r \mapsto \pi^i_k(r)$ 
restricts to a bijection $\pi^{ij}_{kj}$ of
$R_{ij}$ onto $R_{kj}$ and to a 
bijection $\pi^{ji}_{jk}$ of
$R_{ji}$ onto $R_{jk}$.  
For $r \in R_{12}$ and $s \in R_{23}$ 
one has $r\cap s \leq r \cap _2 =0$  and
$(r+s)\cap b= (r+s) \cap (a_1 +a_3)$,
$r+s +b=1 $ where $b=\sum_{i\neq 1,3} a_i$. 
For $r \in R_{12}$ define  $r_{ij}$ as in the proof of Lemma \ref{fex}.  
Then  the definition  $\otimes(s,r,\bar a)=(r_{12} +s_{23})\cap b$ 
of multiplication 
   is given by a term for which the $x,y,\bar z \mapsto
s,r,\bar a$ with $r,s \in R_{12}$ are admissible.   
To obtain the same kind of term for difference, put
\[
r \ominus s = \bigl( [(s_{13} +a_2+a^+) \cap (r+a_{23})] + a_3+a_4\bigr) \cap
(a_1+a_2) .\]
Thus,
the lattice terms and equations  used in the proof of
Lemma \ref{fex} and Lemma~\ref{5}
can be modified to become admissible.
That $\Gamma_{\bar a}$ is an isomorphism (of rings),
is  shown by easy Linear Algebra calculations (cf. \cite{neu}).
(iii) follows, immediately.
 \end{proof}

\section*{Acknowledgments}
\addcontentsline{toc}{section}{Acknowledgment}
The second author was supported by 
the \emph{German Academic Exchange Service} (DAAD)
under codenumber \texttt{91532398-50015537}.

\end{document}